\documentclass[10pt]{article}
\usepackage{amsmath}
\usepackage{amsfonts}
\usepackage{amssymb}
\usepackage{amsmath}
\usepackage{amsthm}
\usepackage{latexsym}
\usepackage{graphicx}

\usepackage{epstopdf}
\usepackage{color}
\usepackage[english]{babel}
\usepackage{float}

\def \C {\mathbb{C}}
\def \R {\mathbb{R}}
\def \H {\mathbb{H}}
\def \N {\mathbb{N}}
\def \bP {\overline{P}^{\prime}}
\def \bQ {\overline{Q}^{\prime}}
\def \mn {\mu^{(N)}}

\newtheorem{theorem}{Theorem}[section]

\newtheorem{corollary}[theorem]{Corollary}
\newtheorem{definition}[theorem]{Definition}

\newtheorem{lemma}[theorem]{Lemma}
\newtheorem{proposition}[theorem]{Proposition}

\begin{document}
\title{Prescribing the $\overline{Q}^{\prime}$-Curvature on Pseudo-Einstein CR 3-Manifolds}

\author{Ali Maalaoui$^{(1)}$}
\addtocounter{footnote}{1}
\footnotetext{Department of mathematics and natural sciences, American University of Ras Al Khaimah, PO Box 10021, Ras Al Khaimah, UAE. E-mail address:
{\tt{ali.maalaoui@aurak.ac.ae}}}

\maketitle
\vspace{5mm}

{\noindent\bf Abstract} {\small In this paper we study the problem of prescribing the $\bQ$-curvature on pseudo-Einstein CR 3-manifolds. In the first stage we study the problem in the compact setting and we show that under natural assumptions, one can prescribe any positive CR pluriharmonic function. In the second stage we study the problem in the non-compact setting of the Heisenberg group. Under mild assumptions on the prescribed function, we prove the existence of a one parameter family of solutions. In fact, we show that one can find two kinds of solutions: normal ones that satisfy an isoperimetric inequality and non-normal ones that have a biharmonic leading term.}
\vspace{8mm}


\vspace{4mm}


\section{Introduction and Main results}
The $Q^{\prime}$-curvature and the $P^{\prime}$-operator play an important role in the study of the geometry of three dimensional manifolds. In fact, the pair $(Q^{\prime},P^{\prime})$ is the parallel of the pair $(Q,P_{4})$ for 4-dimensional conformal manifolds. Indeed, from the correspondence between conformal and CR geometry induced by the Fefferman metric \cite{Feff}, one can construct a pair $(Q,P_{\theta})$ such that under a conformal change of the contact form $\theta \to \hat{\theta}=e^{2u}\theta$, one has
$$P_{\theta}u+Q_{\theta}=Q_{\hat{\theta}}e^{4u}$$
where the Paneitz operator $P_{\theta}=(\Delta_{b})^{2}+T^{2}+lot$. Unfortunately, this construction has two issues. The first one is from an analytical point of view, since the operator $P_{\theta}$ has a huge kernel containing the space of CR pluriharmonic functions $\mathcal{P}$ and its fundamental solution has a leading term of $(\ln|xy^{-1}|)^{2}$ (with $M$ seen as locally diffeomerphic to the Heisenberg group $\H^{1}$). The second issue is that the total $Q$-curvature is always zero \cite{Hir2}, hence it does not provide any extra geometric information compared to the case of the 4-dimensional conformal manifolds where one has
$$\int_{M}Qdv_{g}+\frac{1}{8}\int_{M}|W_{g}|^{2}dv_{g}=4\pi^{2}\chi(M).$$
In \cite{Bran}, the authors, provide a substitute pair, in odd dimensional spheres $(P^{\prime},Q^{\prime})$ where $P^{\prime}$ is a Paneitz type operator in order to prove a sharp Onofri inequality in the CR setting. In dimension $3$, the $P^{\prime}$-operator satisfies $P^{\prime}=4(\Delta_{b})^{2}+lot$ and is defined on the space of pluriharmonic functions and the $Q^{\prime}$-curvature is defined implicitely so that
$$P^{\prime}_{\theta}u+Q^{\prime}_{\theta}-\frac{1}{2}P_{\theta}(u^{2})=Q^{\prime}_{\hat{\theta}}e^{2u},$$
This can be also stated as 
\begin{equation}\label{eq1}
P^{\prime}_{\theta}u+Q^{\prime}_{\theta}=Q^{\prime}_{\hat{\theta}}e^{2u}\text{ mod }\mathcal{P}^{\perp}.
\end{equation}
This was extended in \cite{CaYa} to the case of pseudo-Einstein three dimensional CR manifolds. Contrary to the $Q$-curvature, the total $Q^{\prime}$-curvature is not always zero and it is invariant under the conformal change of the contact structure. In fact, it is proportional to the Burns-Epstein invariant $\mu(M)$  (see \cite{BE} when $T^{1,0}$ is trivial then extended in \cite{CL} ). In particular, as shown in \cite{CaYa}, if $(M,J)$ is the boundary of a strictly pseudo-convex domain $X$, then
$$\int_{M}Q^{\prime}\theta\wedge d\theta=16\pi^{2}\Big(\chi(X)-\int_{X}(c_{2}-\frac{1}{3}c_{1}^{2})\Big),$$
where $c_{1}$ and $c_{2}$ are the first and second Chern forms of the K\"{a}hler-Einstein metric on $X$ obtained by solving Fefferman's equation.\\

Because of the issue of solving orthogonally to the infinite dimensional space $\mathcal{P}^{\perp}$, Case, Hsiao and Yang \cite{CaYa1}, studied another quantity that has similar properties to the $Q^{\prime}$-curvature and that comes from the projection of equation $(\ref{eq1})$ on to the space $\mathcal{P}$. In fact, the $P^{\prime}$-operator as defined in \cite{Bran}, is only defined after projection on $\mathcal{P}$, but in \cite{CaYa1}, the authors show extra analytical properties of this projected operator. Indeed, if we let $\Gamma: L^{2}\to \mathcal{P}$ be the orthogonal projection and we let $\bP=\Gamma\circ P^{\prime}$, then in \cite{CaYa1}, the authors study the equation 
$$\bP u+\bQ=\lambda e^{2u} \text{ mod } \mathcal{P}^{\perp}.$$
The quantity $\bQ$ is the projection of $Q^{\prime}$ on $\mathcal{P}$, that is, $\bQ=\Gamma \circ Q^{\prime}$.\\
In this paper we continue the study of the problem of prescribing the $\bQ$-curvature, under conformal change of the contact structure on pseudo-Einstein CR manifolds. Namely, given a function $Q\in \mathcal{P}$, we want to solve the problem
\begin{equation}\label{eq2}
\bP u+\bQ=Qe^{2u} \text{ mod } \mathcal{P}^{\perp}.
\end{equation}
Naturally, this is equivalent to 
$$\bP u+\bQ =\Gamma(Qe^{2u}).$$
Notice that if $u$ solves $(\ref{eq2})$, then for $\tilde{\theta}=e^{u}\theta$, one has $\bQ_{\tilde{\theta}}=Q$. Ineed, one needs to make clear distinctions between the different projections. That is, $\bQ$ is the orthogonal projection of $Q^{\prime}$ on $\mathcal{P}$ with respect to the $L^{2}$-inner product induced by $\theta$, while $\bQ_{\tilde{\theta}}$ is the orthogonal projection of $Q^{\prime}_{\tilde{\theta}}$ with respect to the $L^{2}$-inner product induced by $\tilde{\theta}$. In particular $\phi\in \mathcal{P}_{\theta}$ if and only if $\phi \in \mathcal{P}_{\tilde{\theta}}$ and $\psi \in \mathcal{P}^{\perp}_{\theta}$ if and only if $e^{-2u}\psi \in \mathcal{P}^{\perp}_{\tilde{\theta}}$. So if we write $\Gamma_{u}$ the orthogonal projection induced by $\tilde{\theta}$, we have $\Gamma_{u}(Q^{\prime}_{\tilde{\theta}})=Q$.

Our main result can be formulated as follows:

\begin{theorem}\label{mthm}
Let $M$ be a three dimensional compact pseudo-Einstein manifold such that $\bP$ is positive and $\ker \bP =\R$ . Consider $Q\in C^{\infty}(M)$ such that $Q>0$ and assume that $\int_{M}\overline{Q}'dv<16\pi^{2}$, then there exists $u\in \mathcal{P}$ such that  $$P'u+Q'= Qe^{2u}\text{ mod } \mathcal{P}^{\perp}.$$
In particular, the contact form $\hat{\theta}=e^{u}\theta$ satisfies $\bQ_{\hat{\theta}}=\Gamma_{u} \circ Q$.
\end{theorem}
We recall that in \cite{CaYa}, the authors show that the non-negativity of the Paneitz operator $P_{\theta}$ and the positivity of the CR-Yamabe invariant imply that $\bP$ is non-negative and $\ker \bP=\R$. Moreover, $\int_{M}Q^{\prime}dv=\int_{M}\bQ dv\leq 16\pi^{2}$ with equality if an only if $(M,J,\theta)$ is the standard sphere. In fact, the previously stated assumptions have very strong geometric implications, namely, they imply that the $(M,J,\theta)$ is embeddable as proved in \cite{Chan}. We also point out some similarities between our result and the work in \cite{QQ}.\\

Our strategy follows an idea from statistical mechanics introduced by Messer and Spohn \cite{MS}, then extended to logarithmic potentials by Kiessling in \cite{Kei}. This method was used in the problem of prescribing the scalar curvature in \cite{CK} and then the problem of prescribing the $Q$-curvature with conical singularities in \cite{Ma}. This will be introduced in Section 2.2. In fact, Theorem \ref{mthm}, will be a direct corollary from the more general result stated in Theorem \ref{pthm}.\\

In section 4, we  consider the case of the Heisenberg group. Since the space is not compact, we will be assuming the following:
given a function $K\in \ker P^{\prime}\cap \ker P$ and $Q\in C^{\infty}(\H)$ such that
\begin{itemize}
\item[a)] For all $0<q<4$, we have $\int_{B_{1}(x)}\frac{Q(y)e^{2K(y)}}{|xy^{-1}|^{q}}dy\to 0$ as $x\to \infty$.
\item[b)]There exists $s\geq 0$ such that $\int_{\H}Q(x)e^{2K(x)}|x|^{s}dx<\infty.$
\end{itemize} 
Then we have the following result
\begin{theorem}\label{thmh}
If $Q\in C^{\infty}(\H)$ satisfies $a)$ and $b)$, then there exists a one parameter family $u_{\beta}$, with $\beta\in (0,8)$, of solutions to
$$4(\Delta_{b})^{2}u=Q(x)e^{2u} \text{ mod } \mathcal{P}^{\perp},$$
with $u(x)=\frac{1}{2}K(x)-\frac{\beta \gamma}{2}\ln|x|+o(\ln|x|)$.
\end{theorem}
We recall that the contact form $e^{u}\theta_{0}$ is said to be normal, (see \cite{WY}), if
$$u(x)=\gamma \int_{\H}\ln\frac{|y|}{|xy^{-1}|}Q(y)e^{2u(y)}dy+C,$$
where $C$ is a constant. In particular, if $K$ is not constant in the above theorem, then $e^{u}\theta_{0}$ is not normal. Hence, Theorem \ref{thmh} provides us with a families of non-normal contact forms. On the other hand, a direct consequence of the result in \cite{WY}, is
\begin{corollary}\label{cor}
Under the same assumptions as in Theorem \ref{mthm}, taking $K$ to be constant, the one parameter family $u_{\beta}$ gives rise to contact forms $\theta_{\beta}=e^{\beta u}\theta_{0}$, satisfying the isoperimetric inequality, where $\theta_{0}$ is the standard contact form on $\H$. That is for any bounded domain $\Omega$ with smooth boundary
$$Vol_{\theta_{\beta}}(\Omega)\leq C_{\beta}Area_{\theta_{\beta}}(\partial \Omega)^{\frac{4}{3}},$$
where $C_{\beta}$ depends on $Q$ and $\beta$.
\end{corollary}
As we will see in Section 4, for $K$ constant, the family of solutions $u_{\beta}$ is normal and has total $\bQ$-curvature equal to $\frac{\beta}{2\gamma}$. Since $\beta<8$, we have that $\int_{\H} Qe^{2u}<16\pi^{2}$, hence, the procedure in \cite{WY} can be applied to show that $e^{2u}$ is an $A_{1}$ weight.
\bigskip

\noindent
{\bf Acknowledgement}
The author wants to express his gratitude to Prof. Paul Yang for the fruitful conversations and insight that helped improve this paper. \\

\section{Preliminaries and Setting}
\subsection{Pseudo-Hermitian geometry}

We will closely follow the notations in \cite{CaYa}. Let $M^3$ be a smooth, oriented three-dimensional manifold.  A CR structure on $M$ is a one-dimensional complex subbundle $T^{1,0}\subset T_{\C}M:= TM\otimes\C$ such that $T^{1,0}\cap T^{0,1}=\{0\}$ for $T^{0,1}:=\overline{T^{1,0}}$.  Let $H=Re T^{1,0}$ and let $J\colon H\to H$ be the almost complex structure defined by $J(Z+\bar Z)=i(Z-\bar Z)$, for all $Z\in T^{1,0}$. The condition that $T^{1,0}\cap T^{0,1}=\{0\}$ is equivalent to the existence of a contact form $\theta$ such that $\ker \theta =H$. We recall that a 1-form $\theta$ is said to be a contact form if $\theta\wedge d\theta$ is a volume form on $M^{3}$. Since $M$ is oriented, a contact form always exists, and is determined up to multiplication by a positive real-valued smooth function. We say that $(M^3,T^{1,0}M)$ is strictly pseudo-convex if the Levi form $d\theta(\cdot,J\cdot)$ on $H\otimes H$ is positive definite for some, and hence any, choice of contact form $\theta$. We shall always assume that our CR manifolds are strictly pseudo-convex.

Notice that in a CR-manifold, there is no canonical choice of the contact form $\theta$. A pseudohermitian manifold is a triple $(M^3,T^{1,0}M,\theta)$ consisting of a CR manifold and a contact form. The Reeb vector field $T$ is the vector field such that $\theta(T)=1$ and $d\theta(T,\cdot)=0$. The choice of $\theta$ induces a natural $L^{2}$-dot product $\langle\cdot,\cdot\rangle$, defined by
$$\langle f,g\rangle =\int_{M}f(x)g(x)\theta\wedge d\theta.$$

 A $(1,0)$-form is a section of $T_{\C}^\ast M$ which annihilates $T^{0,1}$.  An admissible coframe is a non-vanishing $(1,0)$-form $\theta^1$ in an open set $U\subset M$ such that $\theta^1(T)=0$.  Let $\theta^{\bar 1}:=\overline{\theta^1}$ be its conjugate.  Then $d\theta=ih_{1\bar 1}\theta^1\wedge\theta^{\bar 1}$ for some positive function $h_{1\bar 1}$.  The function $h_{1\bar 1}$ is equivalent to the Levi form.  We set $\{Z_1,Z_{\bar 1},T\}$ to the dual of $(\theta^{1},\theta^{\bar 1},\theta)$. The geometric structure of a CR manifold is determined by the  connection form $\omega_1{}^1$ and the torsion form $\tau_1=A_{11}\theta^1$ defined in an admissible coframe $\theta^1$ and is uniquely determined by
\begin{equation*}
\left\{\begin{array}{ll}
d\theta^1 = \theta^1\wedge\omega_1{}^1 + \theta\wedge\tau^1, \\
\omega_{1\bar 1} + \omega_{\bar 11} = dh_{1\bar 1},
\end{array}
\right.
\end{equation*}
where we use $h_{1\bar 1}$ to raise and lower indices.  The connection forms determine the pseudohermitian connection $\nabla$, also called the Tanaka-Webster connection, by
\[ \nabla Z_1 := \omega_1{}^1\otimes Z_1. \] 
The scalar curvature $R$ of $\theta$, also called the Webster curvature, is given by the expression
 \[ d\omega_1{}^1 = R\theta^1\wedge\theta^{\bar 1} \mod\theta . \]
\begin{definition}
A real-valued function $w\in C^\infty(M)$ is CR pluriharmonic if locally $w=Re f$ for some complex-valued function $f\in C^\infty(M,\C)$ satisfying $Z_{\bar 1}f=0$. 
\end{definition}
Equivalently, \cite{Lee}, $w$ is a CR pluriharmonic function if
\[ P_{3} w:=\nabla_1\nabla_1\nabla^1 w + iA_{11}\nabla^1 w = 0 \]
for $\nabla_1:=\nabla_{Z_1}$. We denote by $\mathcal{P}$ the space of all CR pluriharmonic functions. Let $\Gamma: L^{2}(M)\to L^{2}(M)\cap \mathcal{P}$ be the orthogonal projection on the space of pluriharmonic functions. If $S:L^{2}(M)\to \ker \bar\partial_{b}$ denotes the Szego kernel, then 
$$\Gamma=S+\bar S+F,$$
where $F$ is a smoothing kernel as shown in \cite{H}. 
The Paneitz operator $P_\theta$ is the differential operator
\begin{align*}
P_\theta(w) & := 4\text{div}(P_{3}w) \\
& = \Delta_b^2w + T^2 - 4\text{Im} \nabla^1\left(A_{11}\nabla^1f\right)
\end{align*}
for $\Delta_b:=\nabla^1\nabla_1+\nabla^{\bar 1}\nabla_{\bar 1}$ the sublaplacian.  In particular, $\mathcal{P}\subset\ker P_\theta$. Hence, $\ker P_{\theta}$ is infinite dimensional. For a thorough study of the analytical properties of $P_{\theta}$ and its kernel, we refer the reader to \cite{H,CCY,CaYa1}. The main property of the Paneitz operator $P_{\theta}$ is that it is CR covariant \cite{Hir2}. That is, if $\hat\theta=e^w\theta$, then $e^{2w} P_{\hat{\theta}}=P_\theta$.
\begin{definition}
\label{pprime}
Let $(M^3,T^{1,0}M,\theta)$ be a pseudohermitian manifold.  The Paneitz type operator $P_{\theta}^\prime\colon\mathcal{P}\to C^\infty(M)$ is defined by
\begin{align}
P_\theta^\prime f & = 4\Delta_b^2 f - 8 \textnormal{Im}\left(\nabla^\alpha(A_{\alpha\beta}\nabla^\beta f)\right) - 4 \textnormal{Re}\left(\nabla^\alpha(R\nabla_\alpha f)\right) \notag\\
& \quad + \frac{8}{3}\textnormal{Re} (\nabla_\alpha R - i\nabla^\beta A_{\alpha\beta})\nabla^\alpha f - \frac{4}{3}f\nabla^\alpha( \nabla_\alpha R - i\nabla^\beta A_{\alpha\beta})\label{pprime1}
\end{align}
for $f\in\mathcal{P}$.
\end{definition}
The main property of the operator $P_{\theta}^\prime$ is its "almost" conformal covariance as shown in \cite{BG,CaYa}. That is if  $(M^3,T^{1,0}M,\theta)$ is a pseudohermitian manifold, $w\in C^\infty(M)$, and we set $\hat\theta=e^w\theta$, then
\begin{equation}
 \label{pprime}
 e^{2w}\hat P_{\theta}^\prime(u) = P_\theta^\prime(u) + P_\theta\left(uw\right)
\end{equation}
for all $u\in\mathcal{P}$.  In particular, since $P_\theta$ is self-adjoint and $\mathcal{P}\subset \ker P_{\theta}$, we have that the operator $P^\prime$ is conformally covariant, mod $\mathcal{P}^\perp$.
\begin{definition}
A pseudohermitian manifold $(M^3,T^{1,0}M,\theta)$ is pseudo-Einstein if $\nabla_\alpha R - i\nabla^\beta A_{\alpha\beta}=0$.
\end{definition}
 Moreover, if $\theta$ induces a pseudo-Einstein structure then $e^{u}\theta$ is pseudo-Einstein if and only if $u\in \mathcal{P}$. The definition above was stated in \cite{CaYa}, but it was implicitly mentionned in \cite{Hir2}. In particular, if $(M^3,T^{1,0}M,\theta)$ is pseudo-Einstein, then $P_{\theta}^\prime$ takes a simpler form:
$$P_\theta^\prime f = 4\Delta_b^2 f - 8 \text{Im}\left(\nabla^1(A_{11}\nabla^1 f)\right) - 4\text{Re}\left(\nabla^1(R\nabla_1 f)\right).$$
\begin{definition}
 \label{qprime}
 Let $(M^3,T^{1,0}M,\theta)$ be a pseudo-Einstein manifold. The $Q^\prime$-curvature is the scalar quantity defined by
 \begin{equation}
  \label{qprime1}
  Q_\theta^\prime = 2\Delta_b R - 4 |A|^2 + R^2.
 \end{equation}
\end{definition}
The main equation that we will be dealing with is the change of the $Q^{\prime}$-curvature under confrmal change. Let $(M^3,T^{1,0}M,\theta)$ be a pseudo-Einstein manifold, let $w\in\mathcal{P}$, and set $\hat\theta=e^w\theta$.  Hence $\hat\theta$ is pseudo-Einstein. Then \cite{BG,CaYa}
\begin{equation}
 e^{2w} Q_{\hat{\theta}}^\prime = Q_\theta^\prime + P_\theta^\prime(w) + \frac{1}{2}P_\theta\left(w^2\right) .
\end{equation}
In particular, $Q_\theta^\prime$ behaves as the $Q$-curvature for $P_\theta^\prime$, mod $\mathcal{P}^\perp$.
To summarize the similarities between the 3-dimensional pseudo-Einstein manifolds and 4-dimensional Riemannian manifolds, we present the following table:
\begin{center}
\begin{tabular}{ c | c }
Conformal 4- manifold & Pseudo-Einstein 3-manifold \\
\hline \hline
 & \\
$(M^{4},g)$  & $(M^{3},\theta,J)$ \\ 
 & \\
$e^{2u}g$ & $e^{u}\theta$ ; $u$ CR pluriharmonic \\
 & \\
$P_{g}=\Delta_{g}^{2}+\textnormal{div}(\frac{2}{3}R-2Ric)du$ & $P_{\theta}'=4\Delta_{b}^{2}-8\textnormal{Im}(A_{11}u_{\bar{1}})_{\bar{1}}-4\textnormal{Re}(Ru_{1})_{\bar{1}}$\\
 & \\
$Q=-\frac{1}{12}(\Delta R-R^{2}+3|Ric|^{2})$ & $Q'=2\Delta_{b}R-4|A|^{2}+R^{2}$\\
 & \\
$\int_{M}Q_{g}+\frac{1}{8}|W_{g}|^{2}dv_{g}=4\pi^{2}\chi(M)$ & $\int_{M}Q'dv_{\theta}=-\frac{\mu(M)}{16\pi^{2}}$\\
\hline
\end{tabular}
\end{center}

Since we are working modulo $\mathcal{P}^{\perp}$ it is convenient to project the previously defined quantities on $\mathcal{P}$. So we define the operator $\bar P_{\theta}^{\prime}=\Gamma \circ P_{\theta}^{\prime}$ and the $\bar Q^{\prime}$-curvature by $\bar Q^{\prime}_{\theta}=\Gamma( Q^{\prime}_{\theta})$. Notice that
$$\int_{M}Q^{\prime}\theta\wedge d\theta =\int_{M}\overline{Q}^{\prime}_{\theta} \theta\wedge d\theta.$$
Moreover, the operator $\bP_{\theta}$ has many interesting analytical properties. Indeed, $\bP_{\theta}:\mathcal{P}\to \mathcal{P}$ is an elliptic pseudo-differential operator (see \cite{CaYa1}) and if we assume that $\ker \bP_{\theta}=\R$, then its Green's function $G$ satisfies
$$\bP_{\theta}G(\cdot,y)=\Gamma(\cdot,y)-\frac{1}{V},$$
where $V=\int_{M}\theta\wedge d\theta$ is the volume of $M$. Moreover, 
$$G(x,y)=-\frac{1}{4\pi^{2}}\ln(|xy^{-1}|)+\mathcal{K}(x,y),$$
where $\mathcal{K}$ is a bounded kernel as proved in \cite{CCY2}. We want also to clarify the relation between $\bP_{\theta}$ and $\bP_{\hat{\theta}}$ for $\hat{\theta}=e^{u}\theta$. If $\Gamma$ is the $L^{2}$-orthogonal projection, induced by the contact form $\theta$, on $\mathcal{P}$  and $\Gamma_{u}$ the one induced by $\hat{\theta}$, then
$$\bP_{\hat{\theta}}=\Gamma_{u}\circ(e^{-2u}\bP_{\theta}).$$

 From now on we will always assume that $\ker \bP=\R$ and that $\bP$ is non-negative. We will be using a particular solution, $U$, to the problem:
$$\bP U(\cdot,y)=\Gamma(\cdot,y)-\frac{\overline{Q}'}{\int_{M}\overline{Q}'}.$$
One can, then, write $U(x,y)=G(x,y)+H(x)+H(y)$ where $G$ is the Green's function of $\bP$ and $H\in \mathcal{P}$ is the solution to the problem $$\overline{P}'H=\frac{1}{V}-\frac{\overline{Q}'}{\int_{M}\overline{Q}'dx},$$
It is easy to check that, locally, $$U(x,y)=-\gamma \ln|xy^{-1}|+\tilde{\mathcal{H}}(x,y),$$
where $\gamma=\frac{1}{4\pi^{2}}$.

The proof of Theorem \ref{mthm} will be a direct consequence of the following

\begin{theorem}\label{pthm}
We fix a smooth function $Q$ such that $Q(x)>0$ on $M$. For every $\beta\in [0,\frac{8}{\gamma})$, there exist $\rho_{\beta} \in L^{p}(M)$ for all $1\leq p<\infty$, solving the following fixed point problem:
$$\rho_{\beta}(x)=\frac{Q(x)\exp{\left(\beta \int_{M}U(x,y)\rho_{\beta}(y)dy\right)}}{\int_{M}Q(x)\exp{\left(\beta \int_{M}U(x,y)\rho_{\beta}(y)dy\right)}dx}.$$
\end{theorem}

The idea of the proof of the previous result follows a procedure introduced by Messer and Spohn \cite{MS} for the a smooth interaction potential. This method was then developed by Kiessling \cite{Kei,Kei2,Kei3}. The method mainly consists of studying the typical distribution of a family of particles inside a set, that interact through a given Hamiltonian. In our case it will be $U$. In order to develop this method, we need some probabilistic background.

\subsection{Overview of the probabilistic method}
We first define the Hamiltonian, or the potential, of $N$ particles in the manifold $M$. That is, given $N\in \mathbb{N}$ and $x_{1},\cdots,x_{N}\in M$, the Hamiltonian $U^{(N)}$ is defined by
$$U^{(N)}(x_{1},\cdots,x_{N})=\frac{1}{2(N-1)}\sum_{1\leq i\not=j\leq N}U(x_{i},x_{j})=\frac{1}{N-1}\sum_{1\leq i<j\leq N}U(x_{i},x_{j}).$$ 
We now introduce some probabilistic tools. For each $N\in \mathbb{N}$, denote the probability measures on $M^N$ by $P(M^N)$. For a probability measure $\varrho^{(N)}\in P(M^N)$, denote the associated Radon measure by $\hat{\varrho}^{(N)}$ and by this we mean, its action on functions, that is $$\hat{\varrho}^{(N)}(f)=\int_{M^{N}
}f(y)\varrho(dy).$$
A measure $\mu^{(N)}\in P(M^{N})$ is called absolutely continuous with respect to a measure $\varrho^{(N)}\in P(M^{N})$, written $d\mu^{(N)}<<d\varrho^{(N)}$, if there exists a positive $d\varrho^{(N)}$-integrable function $f(x_1, ..., x_N)$, called the density of $\mu^{(N)}$ with respect to $\varrho^{(N)}$, such that $d\mu^{(N)} = f(x_1, ..., x_N) d\varrho^{(N)}$. By $P^{s}(M^{N})$ we mean the space of exchangeable probabilities, i.e. the subset of $P(M^{N})$ whose elements are permutation symmetric in $x_1$, ..., $x_N \in M$. The $n^{th}$ marginal measure of $\varrho^{(N)}\in P^{s}(M^{N})$, $n < N$, is an element of $P^{s}(M^{n})$, given by integrating $\varrho^{(N)}$ with respect to $N-n$ variable. More precisely, given a measurable set $A\subset M^{n}$, the $n^{th}$ marginal $\varrho^{(N)}_{n}$ is given by 
$$ \varrho^{(N)}_{n} (A)=\varrho^{(N)}(A\times M^{(N-n)}).$$
We let $\Omega=M^{\mathbb{N}}$ be the set of sequences with values in $M$. To $\varrho\in P(M)$ we assign the energy functional defined by
\begin{equation}
\mathcal{E}(\varrho) \equiv {1\over 2} \hat{\varrho}^{\otimes 2}(U(x,y))
= {1\over 2} \int_{M}\int_{M} U(x,y) \varrho(dx)\varrho(dy) ,
\end{equation}
whenever the integral on the right exists. We denote by $P_{\mathcal{E}}(M)$ the subset of $P(M)$ for which $\mathcal{E}(\varrho)$ exists. For $\mu\in P^s(\Omega)$ the mean energy of $\mu$ is defined by 
\begin{equation}
e(\mu) = \lim_{n\to \infty}\frac{1}{n}\hat{\mu}_{n}(U^{(n)})={1\over 2} \hat{\mu}_2(U(x,y)),
\end{equation}
whenever the integral on the right exists. Using the decomposition measure introduced by [HS], one has the following proposition: 
\begin{proposition}
The mean energy of $\mu$, is well defined for those $\mu$ whose decomposition measure $\nu(d\varrho|\mu)$ is concentrated on $P_\mathcal{E}(M)$, and in that case it is given by 
\begin{equation}
e(\mu) = \int_{P_{\mathcal{E}}(M)} \nu(d\varrho|\mu)\mathcal{E}(\varrho).
\end{equation}
\end{proposition}
In our setting, we define the measure 
\begin{equation}\label{tau}
\tau (dx)=Q(x)dx,
\end{equation}
and we set $\mathcal{M}^{(1)}=\int_{M}Q(y)dy.$ Thus one can define the probability measure $\mu^{(1)}(dx)=\frac{1}{\mathcal{M}^{(1)}} \tau(dx)$. Next, we define the micro-canonical ensemble, \cite{Elli}, by
\begin{equation}\label{min}
\mu^{(N)}=\frac{1}{\mathcal{M}^{(N)}(\beta)}\exp{\left(\beta \frac{1}{N-1}\sum_{1\leq i<j \leq N}U(x_{i},x_{j})\right)} \prod_{1\leq l\leq N}\tau(dx_{l}),
\end{equation}
where $\mathcal{M}^{(N)}(\beta)$ is a normalizing constant making $\mu^{(N)}$ a probability measure. That is
$$\mathcal{M}^{(N)}(\beta)=\int_{M^{N}}\exp{\left(\beta \frac{1}{N-1}\sum_{1\leq i<j \leq N}U(x_{i},x_{j})\right)} \prod_{1\leq l\leq N}\tau(dx_{l}).$$

For each $\varrho^{(N)}(dx_1...dx_N) \in P\bigl(M^{N}\bigr)$, its entropy with respect to the probability measure $\mu^{(1)}(dx_1)\otimes ...\otimes\mu^{(1)}(dx_N)\equiv \mu^{(1) \otimes N}(dx_1 ...dx_N)$ is defined by
\begin{equation}\label{ent}
\mathcal{S}^{(N)}\left(\varrho^{(N)}\right) = - \int_{M^{N}} \ln \left( {d\varrho^{(N)}\quad \over d\mu^{(1)\otimes N}} \right)
 \varrho^{(N)}(dx_1 ... dx_N)
\end{equation}
if $\varrho^{(N)}$ is absolutely continuous with respect to $d\tau^{\otimes N}$, and provided the integral exists. In all other cases, $\mathcal{S}^{(N)}\left(\varrho^{(N)}\right) = -\infty$.
In particular, if $\mu_n$ is the $n^{th}$ marginal of a measure  $\mu \in P^s(\Omega)$, then the entropy of $\mu_n$, $n \in\{1,...\}$, is given by $\mathcal{S}^{(n)}(\mu_n)$, where $\mathcal{S}^{(n)}$ is defined as in (\ref{ent}) with $\varrho^{(n)} = \mu_n$. We also define $\mathcal{S}^{(0)}(\mu_0) = 0$. \\

After having defined the entropy function, we now state some of its classical properties. We refer the reader to \cite{Kei3} for the details of the proofs. For each $\mu \in P^s(\Omega)$, the sequence $n\mapsto \mathcal{S}^{(n)}(\mu_n)$ enjoys the following

\begin{proposition}
\textbf{Non-positivity }\\
For all $n$,
$$
  \mathcal{S}^{(n)}(\mu_n)\leq 0.  
$$

\textbf{ Monotonic decrease}\\
If $n< n_{1}$, then
$$
  \mathcal{S}^{(n_{1})}(\mu_{n_{1}})\leq \mathcal{S}^{(n)}(\mu_n) .
$$ 

\textbf{Strong sub-additivity}
For $n_{1},\, n_{2} \leq n$, and with $\mathcal{S}^{(-m)}(\mu_{-m}) \equiv 0$ for $m>0$,
$$
  \mathcal{S}^{(n)}(\mu_n) \leq \ \mathcal{S}^{(n_{1})}(\mu_{n_{1}}) 
	+ \mathcal{S}^{(n_{2})}(\mu_{n_{2}}) + \mathcal{S}^{(n-n_{1}-n_{2})}(\mu_{n - n_{1} - n_{2}})
      - \mathcal{S}^{(n_{1} + n_{2} - n)}(\mu_{n_{1} + n_{2} - n}). $$
\end{proposition}

As a consequence of the sub-additivity of $\mathcal{S}^{(n)}(\mu_n)$, the limit 
$$ \mathcal{S}(\mu) = \lim_{n\to\infty} \ {1\over n} \mathcal{S}^{(n)}(\mu_n)\, $$
exists whenever $\inf_n\, n^{-1} \mathcal{S}^{(n)}(\mu_n) > -\infty$; otherwise $\mathcal{S}(\mu) = -\infty$. The quantity $\mathcal{S}(\mu)$ is called the mean entropy of $\mu\in P^s(\Omega)$. The mean entropy is an affine function, moreover one has the following representation .
\begin{proposition}
The mean entropy of $\mu$, is given by
$$
\mathcal{S}(\mu) = \int_{P(M)} \nu(d\varrho|\mu) \mathcal{S}^{(1)}(\varrho) .
$$
\end{proposition}

\section{Proof of Theorem 1.1}

\subsection{First properties of the probability measures}

\noindent
\bigskip
We begin investigating our problem by following the approach developed in \cite{Kei3}. First we have the following integrability property.
\begin{proposition}\label{propp}
For $\beta\gamma \in [0, 8)$, the measure $\mu^{(N)}$ satisfies $d\mu^{(N)}<<d\tau^{\otimes N}$, moreover,
the associated density belongs to $L^{p}(M^{N},\tau^{\otimes N})$ for $p\in [1,\infty]$ if $\beta=0$ and $p\in [1, \frac{8}{\beta \gamma})$ if $\beta \gamma \in (0, 8)$, for $N$ big enough.
\end{proposition}
\begin{proof}
Indeed, using the convexity of the exponential function and the symmetry of $U$, we have
\begin{align}
\mathcal{M}^{(N)}(p\beta)&=\int_{M^{N}}\exp\Big(p\beta U^{(N)}(x_{1},\cdots,x_{N})\Big)d\tau^{\otimes N}(x_{1},\cdots,x_{N})\notag\\
&\leq \frac{1}{N}\sum_{ i=1}^{N}\int_{M^{N}}\exp\Big(\frac{p\beta}{2}\frac{N}{N-1}\sum_{j=1,j\not=i}^{N}U(x_{i},x_{j})\Big)d\tau^{\otimes N}(x_{1},\cdots,x_{N})\notag\\
&\lesssim \frac{1}{N}\sum_{ i=1}^{N}\int_{M}\Big(\int_{M}\exp(\frac{p\beta}{2}\frac{N}{N-1}U(x,y))\tau(dy)\Big)^{N-1}\tau(dx)\notag\\
&\lesssim 1+\int_{M}\Big(\int_{B_{x}(1)}\exp\Big(\frac{p\beta\gamma}{2}\frac{N}{N-1}\ln(\frac{1}{|xy^{-1}|})\Big)\tau(dy)\Big)^{N-1}\tau(dx)\notag
\end{align}
It is clear that the integrand is finite, whenever $p\beta \gamma \frac{N}{N-1}<8$.
\end{proof}

\noindent
We set the approximated variational problem by defining the functional $\mathcal{F}^{(N)}_{\beta}$ as follows
$$\mathcal{F}^{(N)}_{\beta}(\varrho^{(N)}):=\mathcal{S}^{(N)}(\varrho^{(N)})+\beta \hat{\varrho}^{(N)}\left(U^{(N)}\right).$$
This functional is well defined on probability measures in $P(M^{N})\cap \cup_{p>1} L^{p}(M^{N},d\mu^{(1)\otimes N})$ that are absolutely continuous with respect to $\tau^{\otimes N}$. We will denote their space by $X_{N}$.

\begin{lemma}\label{min}
For $\beta \gamma\in [0,8)$ the functional $\mathcal{F}_{\beta}^{(N)}$ has a unique maximum and it is achieved by the measure $\mu^{(N)}$. That is
\begin{equation}\label{inf}
\mathcal{F}^{(N)}(\beta):=\sup_{\varrho^{(N)}\in X_{N}}\mathcal{F}_{\beta}^{(N)}(\varrho^{(N)})=\mathcal{F}_{\beta}^{(N)}(\mu^{(N)}).
\end{equation}
Moreover,
\begin{equation}\label{infval}
\mathcal{F}_{\beta}^{(N)}(\mu^{(N)})= \ln \left(\frac{\mathcal{M}^{(N)}(\beta)}{(\mathcal{M}^{(1)})^{N}}\right).
\end{equation}
\end{lemma}
\begin{proof} 

First, notice that $\mathcal{F}_\beta^{(N)}(\mu^{(N)})$ is well defined for $\beta \in [0,\frac{8}{\gamma})$ and an explicit computation gives the equation (\ref{infval}).\\
Now,  
\begin{align}
\mathcal{F}_\beta^{(N)}\left(\varrho^{(N)}\right)& = \beta \int _{M^{N}} U^{(N)} \frac{d\varrho^{(N)}}{d\mu^{(1)\otimes N}}d\mu^{(1)\otimes N}(dx_{1},...dx_{N}) \notag\\
&\quad- \int _{M^{N}} \ln \left(\frac{d\varrho^{(N)}}{d\mu^{(1)\otimes N}} \right) \frac{d\varrho^{(N)}}{d\mu^{(1)\otimes N}}d\mu^{(1)\otimes N}(dx_{1},...dx_{N}).
\end{align}
But 
$$\frac{d\varrho^{(N)}}{d\mu^{(1)\otimes N}}=\frac{(\mathcal{M}^{(1)})^{N}}{\mathcal{M}^{(N)}(\beta)} e^{\beta U^{(N)}}\frac{d\varrho^{(N)}}{d\mu^{(N)}}.$$
Hence, 
\begin{align}
\mathcal{F}_{\beta}^{(N)}(\varrho^{(N)})&=-\int _{M^{N}} \ln \left( \frac{d\varrho^{(N)}}{d\mu^{(N)}}\right)\varrho^{(N)}(dx_{1},...,dx_{N})-\ln \left(\frac{(\mathcal{M}^{(1)})^{N}}{\mathcal{M}^{(N)}(\beta)}\right)\notag\\
&=-\int _{M^{N}} \ln \left( \frac{d\varrho^{(N)}}{d\mu^{(N)}}\right)\varrho^{(N)}(dx_{1},...,dx_{N})+\mathcal{F}_{\beta}^{(N)}(\mu^{(N)}),\notag
\end{align}
and using the fact that $x\ln x \geq x-1$, with equality iff $x=1$, we find that $$\mathcal{F}_\beta^{(N)}(\varrho^{(N)})-\mathcal{F}_\beta^{(N)}(\mu^{(N)}) \leq 0,$$ with equality holding if and only if $\varrho^{(N)} = \mu^{(N)}$. 
\end{proof}
Next, we show a very important property for the sequence $\mathcal{F}^{(N)}(\beta)$.
\begin{proposition}\label{prop1}
Given $\beta <\frac{8}{\gamma}$, the limit
$$\lim_{N\to \infty}\frac{1}{N}\mathcal{F}^{(N)}(\beta)=:f(\beta),$$
exists and is finite.
\end{proposition}
The proof of this proposition will follow from the next two lemmata.
\begin{lemma}\label{lem34}
The sequence $\frac{1}{N}\mathcal{F}^{(N)}(\beta)$ is bounded below and above independently of $N$.
\end{lemma}
{\it Proof:}
For the bound from below, we apply Jensen's inequality to $\mathcal{M}^{(N)}(\beta)$ with the concave function $\ln(\cdot)$. This leads to $$\ln\left(\frac{\mathcal{M}^{(N)}(\beta)}{\left(\mathcal{M}^{(1)}\right)^{N}}\right)\geq \frac{N}{2}\beta\hat{ \mu}^{(1)\otimes 2}(U(x,y))$$
Hence,
$$\frac{1}{N}\mathcal{F}^{(N)}(\beta)\geq \frac{\beta}{2} \hat{\mu}^{(1)\otimes 2}(U(x,y)).$$
The bound from above, can be deduced the exact same way as in Proposition 3.1.
\hfill$\Box$

\begin{lemma}\label{lem35}
The sequence $N\to \mathcal{F}^{(N)}(\beta)$ is sub-additive. That is, if $N=N_{1}+N_{2}$ then
$$\mathcal{F}^{(N)}(\beta)\leq \mathcal{F}^{(N_{1})}(\beta)+\mathcal{F}^{(N_{2})}(\beta).$$
\end{lemma}
{\it Proof:}

We set $N=N_{1}+N_{2}$, then we have
\begin{align}
\mathcal{F}^{(N)}_{\beta}(\mu^{(N)})&=\mathcal{S}^{(N)}(\mu^{(N)})+\frac{\beta}{2} N\hat{\mu}^{(N)}_{2}(U(x,y))\notag\\
&\leq \mathcal{S}^{(N_{1})}(\mu^{(N)}_{N_{1}})+\mathcal{S}^{(N_{2})}(\mu^{(N)}_{N_{2}})+\frac{\beta}{2}(N_{1}+N_{2})\hat{\mu}^{(N)}_{2}(U(x,y))\notag\\
&\leq \mathcal{F}^{(N_{1})}(\beta)+\mathcal{F}^{(N_{2})}(\beta),\notag
\end{align}
where in the first equation, we used the symmetry of $U$ and $\mu^{(N)}$ and in the second inequality the sub-additivity of the entropy $\mathcal{S}$.
\hfill$\Box$

The boundedness from below and the sub-additivity provided by Lemma \ref{lem34} and \ref{lem35}, insure the result of Proposition \ref{prop1}.

\subsection{Integrability}

The objective now is to show compactness (in the weak sense) of the sequence $(\mu^{(N)}_{n})_{N}$. In order to do that, we need to show a uniform $L^{p}$-boundedness for the sequence in question. We claim that
\begin{proposition}\label{propb}
There exists a constant $K(n,\beta \gamma)$ such that
$$\mu^{(N)}_{n}(dx_{1}\cdots dx_{n})\leq K(n,\gamma \beta)\exp{\left(\frac{\beta}{N-1}\sum_{1\leq i<k \leq n}U(x_{i},x_{j})\right)}\tau^{\otimes n}.$$
\end{proposition}
{\it Proof:}
First, we write $(N-1)U^{(N)}=W^{(n)}+W^{(n,N-n)}+W^{(N-n)}$. Here, $W^{(n)}$ is the term involving $(x_{1},\cdots,x_{n})$, $W^{(N-n)}$ is the term involving $(x_{n+1},\cdots, x_{N}),$ and finally the term $W^{(n,N-n)}$ contains the mixed remaining variables. First notice that $\frac{1}{N-1}W^{(n)}\to 0$ as $N\to \infty$, hence $e^{\frac{\beta}{N-1}W^{(n)}}\in L^{p}(M^{n})$ for $N$ big enough.\\
 
Next, we move to the term $W^{(n,N-n)+W^{(N-n)}}$. Indeed, we take $q=\frac{N-1}{2n}$ and $q'=\frac{N-1}{N-1-2n}$ and using H\"{o}lder's inequality we get
\begin{align}
\left\|\exp{\left(\frac{\beta}{N-1}\left[W^{(n,N-n)}+W^{(N-n)}\right]\right)}\right\|_{L^{1}(M^{N})}\leq& \left\|\exp{\left(\frac{\beta}{N-1}W^{(n,n-N)}\right)}\right\|_{L^{q}(M^{N})}\times\notag\\
&\left\|\exp{\left(\frac{\beta}{N-1}W^{(n-N)}\right)}\right\|_{L^{q'}(M^{N})}.\notag
\end{align}
The first integral can be bounded the same way as in Proposition \ref{propp} and the fact that 
\begin{align}
\left\|\exp{\left(\frac{\beta}{N-1}\sum_{k=1}^{n}U(x_{k},x)\right)}\right\|_{L^{q}(M)}^{N-n}&= \left\|\exp{\left(\frac{\beta}{N-1}(-\gamma\sum_{k=1}^{n}\ln|x_{k}x^{-1}|\chi_{B_{1}(x_{k})}+\tilde{H}(x))\right)}\right\|_{L^{q}(M)}^{N-n}\notag\\
&\leq C^{\frac{N-n}{N-1}}\left\|\frac{1}{|x|^{\frac{n\gamma\beta}{N-1}}}\right\|_{L^{q}(B_{1}(0))}^{N-n}\notag\\
&\leq C(n)\left\|\frac{1}{|x|^{\frac{\gamma \beta}{2}}}\right\|_{L^{1}(B_{1}(0))}^{\frac{2n(N-n)}{N-1}}.
\end{align}

Next we deal with the second term, namely $\|\exp(\frac{\beta}{N}W^{(n-N)})\|_{L^{q'}}$, where $q'=\frac{N-1}{N-2n-1}$. This can be written as:
$$\left\|\exp{\left(\frac{\beta}{N-1}W^{(n-N)}\right)}\right\|_{L^{q'}(M^{N})}=\mathcal{M}^{(N-n)}\left(\beta\frac{N-n-1}{N-2n-1}\right)^{1-\frac{2n}{N-1}}.$$
Notice that since $\lim_{N\to \infty}\frac{1}{N}\mathcal{F}^{(N)}(\beta)$ exists, we have that $\mathcal{M}^{(N-n)}\left(\beta \frac{N-n-1}{N-2n-1}\right)^{-\frac{2n}{N-1}}$ is uniformly bounded. Hence, it remains to bound $\mathcal{M}^{(N-n)}\Big(\beta\frac{N-n-1}{N-2n-1}\Big)$. Using Jensen's inequality with respect to the measure $d\mu^{(1)\otimes n}$, we have 
\begin{align}\mathcal{M}^{(N)}(\beta)&\geq \left(\mathcal{M}^{(1)}\right)^{n}\exp{\left(\frac{n(2N-n-1)}{N-1}\beta\mu^{(1)\otimes 2}(U(x,y))\right)}\mathcal{M}^{(N-n)}\left(\frac{N-n-1}{N-1}\beta\right)\notag\\
&\geq C(n,\beta)\mathcal{M}^{(N-n)}\left(\frac{N-n-1}{N-1}\beta\right).\notag
\end{align}
We now consider the density $\rho^{(N-n)}$ defined by $$\rho^{(N-n)}=\frac{\exp{\left(\beta \frac{1}{N-2n-1}W^{(N-n)}\right)}}{\mathcal{M}^{(N-n)}\left(\frac{N-n-1}{N-2n-1}\beta\right)}.$$
We will write $\left\langle X \right\rangle _{N}$ the average of $X$ with respect to the density $\rho^{(N-n)}$ and the measure $\tau^{\otimes(N-n)}$. Therefore, we have
\begin{align}
\frac{\mathcal{M}^{(N-n)}\left(\frac{N-n-1}{N-1}\beta\right)}{\mathcal{M}^{(N-n)}\left(\frac{N-n-1}{N-2n-1}\beta\right)}&=\left \langle \exp{\left(-\frac{2n}{(N-1)(N-2n-1)}\beta W^{(N-n)}\right)} \right \rangle_{N}\notag\\
&\geq \exp{\left(-\frac{2n}{(N-1)(N-2n-1)}\left\langle \beta W^{(N-n)}\right\rangle_{N}\right)}\notag\\
&=\exp{\left(-2n\beta\partial_{\beta}\left(\frac{1}{N-1}\mathcal{F}^{(N-n)}\left(\frac{N-n}{N-2n-1}\beta\right)\right)\right)}.\notag
\end{align}
But recall that since $\beta \mapsto \frac{1}{N}\mathcal{F}^{(N)}(\beta)$ is convex (it is easily verified by taking two derivatives), the function $\beta \mapsto f(\beta)$ is also convex. In particular, its derivative exists almost everywhere and it is non-decreasing. So, for $\beta_{0} \in (\beta, \frac{4}{\gamma})$, we have that
$$\frac{\mathcal{M}^{(N-n)}\left(\frac{N-n-1}{N-1}\beta\right)}{\mathcal{M}^{(N-n)}\left(\frac{N-n-1}{N-2n-1}\beta\right)}\geq C\exp{\left(-2n\beta \partial_{\beta}^{+}(f(\beta_{0}))\right)},$$
and this finishes the proof.
\hfill$\Box$

The previous proposition states that $\mu^{(N)}_{n}$ has a density with respect to $d\tau^{n}$ (or $d\mu^{(1)\otimes n}$), in $L^{p}(M^{n})$ for all $p\geq 1$. In particular the sequence $(\mu^{(N)}_{n})_{N}$ is weakly compact in the space $P(M^{n})\cap L^{p}(M^{n})$. We want to characterize the limit points.
\begin{proposition}\label{pconv}
 Let us consider a weakly convergent subsequence $\mu^{(a(N))}_{n}$ that converges weakly to a limit point say $\mu(\beta)\in P_{s}(\Omega)$. Then the decomposition measure of $\mu(\beta)$ is concentrated at the maximizers of $\mathcal{F}^{(1)}_{\beta}$.
\end{proposition} 
{\it Proof:}
Recall that
\begin{align}
\mathcal{F}_{\beta}(\mu)&=\lim_{n \to \infty}\frac{1}{n}\mathcal{F}_{\beta}^{(n)}(\mu_{n})\notag\\
&=\int_{P_{\mathcal{E}}(M)}\mathcal{S}^{(1)}(\rho)+\frac{\beta}{2}\hat{\rho}^{\otimes 2}(U(x,y))\nu(d\rho|\mu)
\end{align}
In particular, if we set $$A_{\beta}=\sup_{\rho \in P_{\mathcal{E}}(M)}\mathcal{S}^{(1)}(\rho)+\frac{\beta}{2}\hat{\rho}^{\otimes 2}(U(x,y))=\sup_{\rho\in P_{\mathcal{E}}(M)}\mathcal{F}^{(1)}(\beta),$$ then one has
$$\sup_{\mu \in P^{s}(M)}\mathcal{F}_{\beta}(\mu)\leq A_{\beta}.$$
On the other hand, we have
\begin{align}
\mathcal{F}^{(N)}(\beta)&=\mathcal{F}_{\beta}^{(N)}\left(\mu^{(N)}\right)\geq \mathcal{F}_{\beta}^{(N)}\left(\rho^{\otimes N}\right)\notag\\
&\geq N\left(\mathcal{S}^{(1)}(\rho)+\frac{\beta}{2}\rho^{\otimes 2}(U(x,y))\right).
\end{align}
Hence, $$f(\beta)\geq A_{\beta}.$$
Next, we write $\alpha(N)=n\left \lfloor \frac{\alpha(N)}{n}\right \rfloor+m$ and using the sub-additivity of the entropy $\mathcal{S}$, we have 
\begin{align}
\mathcal{S}^{(\alpha(N))}(\mu^{(\alpha(N))})&\leq \left \lfloor \frac{\alpha(N)}{n}\right \rfloor \mathcal{S}^{(n)}(\mu_{n}^{(\alpha(N))})+\mathcal{S}^{(m)}(\mu_{m}^{(\alpha(N))})\notag\\
&\leq \left\lfloor \frac{\alpha(N)}{n}\right \rfloor \mathcal{S}^{(n)}(\mu_{n}^{(\alpha(N))})\notag
\end{align}
Using the upper-semicontinuity of the Entropy, we have 
$$\limsup_{N\to \infty}\mathcal{S}^{(n)}(\mu_{n}^{\alpha(N)})\leq \mathcal{S}^{(n)}(\mu_{n}(\beta)).$$
Hence,
\begin{align}
\limsup_{N\to \infty}\frac{1}{\alpha(N)}\mathcal{S}^{(\alpha(N))}(\mu^{\alpha(N)})&\leq \limsup \frac{1}{\alpha(N)}\left \lfloor \frac{\alpha(N)}{n}\right \rfloor \mathcal{S}^{(n)}(\mu_{n}^{(\alpha(N))})\notag\\
&\leq \frac{1}{n}\mathcal{S}^{(n)}(\mu_{n}(\beta))\notag
\end{align}
Therefore, if we let $n\to \infty$, we have
$$\limsup_{N\to \infty}\frac{1}{\alpha(N)}\mathcal{S}^{(\alpha(N))}(\mu^{\alpha(N)})\leq \mathcal{S}(\mu(\beta))$$
In particular
\begin{align}
f(\beta)&=\limsup \frac{1}{\alpha(N)}\mathcal{F}_{\beta}^{(\alpha(N))}(\mu^{(\alpha(N))})\notag\\
&\leq \mathcal{F}_{\beta}(\mu(\beta))\notag\\
&\leq \sup_{\mu\in P^{s}(\Omega)}\mathcal{F}_{\beta}(\mu)\notag
\end{align}
Therefore, $A_{\beta}=f(\beta)=\mathcal{F}_{\beta}(\mu(\beta))$.

Thus the limiting points concentrate at the maximizers of $A_{\beta}$. Hence, $A_{\beta}=\max_{\rho\in P_{\mathcal{E}}(M)}\mathcal{F}_{\beta}^{(1)}(\rho).$\\
In fact, one can see that the decomposition measure is actually concentrated on measures with density that is in $L^{p}(M)$ for all $p>1$.
\hfill$\Box$

Now to finish the proof of Theorem \ref{pthm}, we notice that as a consequence of Proposition \ref{pconv}, the maximization problem 
$$f(\beta)=\sup \{\mathcal{F}_{\beta}^{(1)}; \rho \in P(M)\cap L^{1}Log(L)(M)\}$$ 
has a solution and thus the solution satisfies the Euler-Lagrange equation 
$$\rho_{\beta}(x)=\frac{Qe^{\beta\int_{M}U(x,y)\rho_{\beta}(y)dy}}{\int_{M}Q e^{\beta\int_{M}U(x,y)\rho_{\beta}(y)dy}dx}.$$
The fact that $\rho_{\beta}\in L^{p}(M)$ follows from the regularity result of the density of the sequence $\mu_{n}^{(N)}$.

\subsection{Proof of the Main result}
Using Theorem \ref{pthm}, we take $u=\frac{\beta}{2}\int_{M}U(x,y)\rho_{\beta}(y)dy+c$, where $c$ is a constant to be determined later. Then we have that
$$\overline{P}'u=\frac{\beta}{2}\Big[\rho_{\beta}^{T}-\frac{\overline{Q}'}{\int_{M}\overline{Q}'dx}\Big],$$
where $\rho_{\beta}=\rho_{\beta}^{T}+\rho_{\beta}^{\perp}$ and $\rho_{\beta}^{T}=\Gamma(\rho_{\beta})$. Thus
$$\overline{P}'u+\frac{\beta}{2}\frac{\overline{Q}'}{\int_{M}\overline{Q}'dx}=\frac{\beta}{2}\lambda Qe^{2u}-\frac{\beta}{2}\rho_{\beta}^{\perp},$$
where $\lambda=\int_{M}Q e^{\beta\int_{M}U(x,y)\rho_{\beta}(y)dy}dx$. Since $\beta \gamma \in [0,8)$, and $\int_{M}\overline{Q}'dx<16\pi^{2}$, one can pick $\beta=2\int_{M}\overline{Q}'dx$  and $e^{2c}=\lambda$, to obtain a solution to 
$$\overline{P}'u+\overline{Q}'= Qe^{2u}\text{ mod } \mathcal{P}^{\perp}.$$

\section{Case of the Heisenberg group}
In this section we will extend the previous result to the non-compact case of the Heisenberg group. Notice that the estimates in the previous section rely on the compactness of the manifold $M$, so we need to adapt them to our new setting. We will be following the procedure developed in \cite{CK} and \cite{Ma} for the Euclidean case.
From now on we fix a "biharmonic" and pluriharmonic function $K$. That is $K$ satisfies
$$(\Delta_{b})^{2}K=0 \quad \text{ and } T^{2}K=0.$$
One such function would be $K(x,y,t)=-(x^{2}+y^{2})$, but also one could think of a more complicated functions. We also consider the following two assumptions on $K$ and $Q$:

\begin{itemize}
\item[a)] For all $0<q<4$, we have $\int_{B_{1}(x)}\frac{Q(y)e^{2K(y)}}{|xy^{-1}|^{q}}dy\to 0$ as $x\to \infty$.
\item[b)]There exists $s\geq 0$ such that $\int_{\mathbb{H}}Q(x)e^{2K(x)}|x|^{s}dx<\infty$
\end{itemize} 
These assumptions will guarantee that the mass does not escape to infinity. An explicit computation done in \cite{WY} shows that the Green's function of the operator $P'$ or $\bP$ has the explicit form $G(x,y)=-\frac{1}{4\pi^{2}}\ln(|xy^{-1}|)$ and 
$$\bP G(\cdot,y)=Re S(\cdot, y)$$
where $S$ is the Szego kernel. Therefore, we will take $U(x,y)=G(x,y)$. For the sake of notation, we will remove the factor $\frac{1}{4\pi^{2}}$ in the definition of $U$. The measure $\tau$ defined in $(\ref{tau})$ will be replaced by
$$\tau(dx)=e^{2K(x)}Q(x)dx.$$
Notice that from the assumption $(b)$, we have that the mass $\mathcal{M}^{(1)}$ of $\tau$ is finite and hence the probability measure $\mu^{(1)}$ is still well defined. The Hamiltonian $U^{(N)}$ then can be written as

$$U^{(N)}(x_{1},\cdots, x_{N})=-\ln(|R^{(N)}|^{\frac{1}{N-1}})$$
where $R^{(N)}=\Pi_{1\leq i<j\leq N}|x_{i}x_{j}^{-1}|$. The definition of the entropy and the energy will remain unchanged. So as in Lemma \ref{min}, we have that $\mathcal{F}_{\beta}^{(N)}$ has a unique minimizer $\mu^{(N)}$ that can be written as
$$\mu^{(N)}=\frac{1}{\mathcal{M}^{(N)}(\beta)}|R^{(N)}|^{-\frac{\beta}{N-1}}.$$
For the well definedness of $\mu^{(N)}$ one needs to show that $\mathcal{M}^{(N)}(\beta)$ is finite.
\begin{lemma}
The measure $\mu^{(N)}$ is absolutely continuous with respect to the measure $\tau^{\otimes N}$. Moreover, $\frac{d\mn}{d\tau^{\otimes N}}\in L^{p}(\H^{N})$ for $p\in [1,\frac{8}{\beta})$, for $N$ large enough.
\end{lemma}
{\it Proof:}
We have for $p\geq 1$
\begin{align}
\int_{\H^{N}}|R^{(N)}|^{-\frac{p\beta}{N-1}}\prod_{1\leq i\leq N}\tau(dx_{i})&\leq \frac{1}{N}\int_{\H^{N}}\sum_{i=1}^{N}\prod_{1\leq j\leq N;j\not=i}|x_{i}x_{j}^{-1}|^{-\frac{pN\beta}{2(N-1)}}\prod_{1\leq i\leq N}\tau(dx_{i})\notag\\
&\leq \int_{\H^{N}}\prod_{2\leq j\leq N}|x_{1}x_{j}^{-1}|^{-\frac{pN\beta}{2(N-1)}}\prod_{1\leq i\leq N}\tau(dx_{i})\notag\\
&\leq \int_{\H} \Big(\int_{\H}|xy^{-1}|^{-\frac{pN\beta}{2(N-1)}}\tau(dy)\Big)^{N-1}\tau(dx)\notag\\
&\leq \sup_{x\in \H}\Big(\int_{\H}|xy^{-1}|^{-\frac{pN\beta}{2(N-1)}}\tau(dy)\Big)^{N-1} \mathcal{M}^{(1)}\notag
\end{align}
where we used the arithmetic-geometric inequality in the second inequality. Now we have that
\begin{align}
\int_{\H}|xy^{-1}|^{-\frac{pN\beta}{2(N-1)}}\tau(dy)&=\int_{B_{1}(x)}|xy^{-1}|^{-\frac{pN\beta}{2(N-1)}}\tau(dy)+\int_{\H\setminus B_{1}(x)}|xy^{-1}|^{-\frac{pN\beta}{2(N-1)}}\tau(dy)\notag\\
&\leq g(x)+\mathcal{M}^{(1)}\notag
\end{align}
but using assumption $(a)$, we have that $g(x)=\int_{B_{1}(x)}|xy^{-1}|^{-\frac{pN\beta}{2(N-1)}}\tau(dy)$ is in $L^{\infty}(\H)$ as long as $\frac{Np}{N-1}<\frac{8}{\beta}$.
\hfill$\Box$

In order to get weak compactness of the measure $\mn$, we need a few Lemmata, including the uniform $L^{p}$ boundedness of the marginals, as in Proposition \ref{propb}.
\begin{lemma}\label{lem1}
Given $\beta \in (0,8)$, there exists two constants $c_{1}$ and $c_{2}$ depending only on $\beta$ such that
$$c_{1}\leq \beta \hat{\mu}_{2}^{(N)}(\ln|xy^{-1}|)\leq \beta \hat{\mu}^{(1)\otimes 2}(\ln|xy^{-1}|)\leq c_{2}.$$
\end{lemma}
{\it Proof:}
For the last inequality, we use the fact that $|xy^{-1}|\leq c(|x|+|y|)\leq c(2+|x|)(2+|y|)$. Then from assumption $(b)$, we have that $$\beta \hat{\mu}^{(1)\otimes 2}(\ln|xy^{-1}|)\leq c_{2}.$$
So we move to the second inequality. We define the function $f_{N}$ by
$$f_{N}(\beta)=-\frac{2}{N}\ln(\hat{\mu}^{(1)\otimes N}(|R^{(N)}|^{-\frac{\beta}{N}}).$$
Using Jensen's inequality, we have that
$$f_{N}(\beta)\leq \frac{2\beta}{N(N-1)}\hat{\mu}^{(1)\otimes N}(\ln(|R^{(N)}|)\leq \beta \hat{\mu}^{(1)}(\ln|xy^{-1}|).$$
On the other hand, notice that
\begin{equation}\label{fn}
-2\mathcal{F}^{(N)}_{\beta}(\mn)=Nf_{N}(\beta).
\end{equation}
Therefore
$$f_{N}(\beta)=\frac{1}{N}(-2\mathcal{S}^{(N)}(\mn)-2\beta \hat{\mu}^{(N)}(U^{(N)}),$$
and by the non-positivity of the entropy, we have
$$\beta \hat{\mu}^{(1)}(\ln|xy^{-1}|)\geq f_{N}(\beta)\geq -\frac{2}{N}\beta \hat{\mu}^{(N)}(U^{(N)})=\beta \hat{\mu}^{(N)}_{2}(\ln|xy^{-1}|).$$
It remains to show the first inequality. Since $\beta \in (0,8)$, there exists $\varepsilon>0$ such that $(1+\varepsilon)\beta \in (0,4)$. By applying Jensen's inequality twice, we have that
$$\mathcal{M}^{(N)}((1+\varepsilon)\beta)\geq \mathcal{M}^{(N)}(\beta)\exp(-\frac{1}{2}N\varepsilon \beta \hat{\mu}_{2}^{(N)}(\ln|xy^{-1}|)).$$
Hence,
$$f_{N}(\beta(1+\varepsilon))\leq f_{N}(\beta)+\varepsilon \beta \hat{\mu}_{2}^{(N)}(\ln|xy^{-1}|).$$
We now consider the function $f_{0}$ defined by 
$$f_{0}(\beta)=-\ln\Big(\sup_{x\in \H}\int_{\H}|xy^{-1}|^{-\frac{\beta}{2}}\mu^{(1)}(dy)\Big).$$
Assumption $(b)$ guaranties that $f_{0}(\beta)$ is well defined and finite and one can easily check that given $\beta\in (0,8)$, then there exists $N_{0}>0$ such that for $N\geq N_{0}$ we have
\begin{equation}\label{f0}
f_{N}(\beta)\geq f_{0}((1+\varepsilon)\beta)+f_{0}(\beta).
\end{equation}
Now from $(\ref{fn})$ and $(\ref{f0})$, we have that
$$f_{0}(\beta)+f_{0}((1+\varepsilon)\beta)\leq f_{N}(\beta)\leq -2A_{\beta}$$
Thus, with $\varepsilon$ even smaller if needed, we have
 $$\beta\hat{\mu}_{2}^{(N)}(\ln|xy^{-1}|)\geq \frac{1}{\varepsilon}(f_{N}((1+\varepsilon)\beta)-f_{N}(\beta))\geq (f_{0}((1+\varepsilon)^{2}\beta)+f_{0}((1+\varepsilon)\beta)-A_{\beta})\geq c_{1}.$$
\begin{lemma}\label{lem2}
Given $\beta \in (0,8)$, there exists $N_{1}>0$ such that for $N\geq N_{1}$, there exists a constant $c_{3}$ depending only on $\beta$ such that
$$\beta \hat{\mu}^{(1)}\otimes \hat{\mu}_{1}^{(N)}(\ln|xy^{-1}|)\leq c_{3}.$$
\end{lemma}
{\it Proof:}
First, we use the inequality $|xy^{-1}|\leq c(|x|+2)(|y|+2)$  to have
$$\hat{\mu}^{(1)}\otimes \hat{\mu}_{1}^{(N)}(\ln|xy^{-1}|)\leq \tilde{c}+\hat{\mu}^{(1)}(\ln(2+|x|))+\hat{\mu}_{1}^{(N)}(\ln(2+|y|)).$$
Assumption $(a)$ yields
$$\hat{\mu}^{(1)}(\ln(2+|x|))\leq C_{1}.$$
Therefore, it remains to bound the second term. First, we have for $\beta'=(1-\frac{1}{N-1})\beta$,
\begin{align}
\hat{\mu}_{1}^{(N)}(\ln(2+|y|))=&\frac{\mathcal{M}^{(N-1)}(\beta')}{\mathcal{M}^{(N)}(\beta)}\int_{H^{N-1}}\frac{|R^{(N-1)}|^{-\frac{\beta'}{N-2}}}{\mathcal{M}^{(N-1)}(\beta')}\times \notag\\
&\Big(\int_{\H}\prod_{i=1}^{N-1}|x_{i}y^{-1}|^{-\frac{\beta}{N-1}}\ln(2+|y|)\tau(dy)\Big)\prod_{i=1}^{N-1}\tau(dx_{i}).\notag
\end{align}
 We fix $s\in (0,s^{*})$, where $s^{*}$ is the sup of all $s>0$ for which $(b)$ holds. Using the inequality $e^{X}+Y\ln(Y)-Y\geq XY$, for 
$$X=e^{s\ln(2+|y|)}$$
and 
$$Y=\frac{1}{s}\int_{\H^{N-1}}\frac{|R^{(N-1)}|^{-\frac{\beta'}{N-2}}}{\mathcal{M}^{(N-1)}(\beta')}\prod_{i=1}^{N-1}|x_{i}y^{-1}|^{-\frac{\beta}{N-1}}\prod_{i=1}^{N-1}\tau(dx_{i}),$$
yields
\begin{align}
C_{N}(\beta)&:=\hat{\mu}_{1}^{(N)}(\ln(2+|y|))-\frac{\mathcal{M}^{(N-1)}(\beta')}{\mathcal{M}^{(N)}(\beta)}\int_{\H}\exp(s(2+|y|))\tau(dy)\notag\\
&\leq -\frac{1}{s}(1+\ln(s)+\beta'\hat{\mu}_{2}^{(N)}(\ln|x-y|))\notag\\
&\leq \tilde{c}_{2}(\beta)\notag
\end{align}
where the last inequality follows from Lemma \ref{lem1}. Clearly, from assumption $(b)$, we have the finiteness of the integral $\int_{\H}\exp(s\ln(2+|y|))\tau(dy)$. Therefore, in order to finish the proof, it is enough to show the $N$-independent bound of the quotient $\frac{\mathcal{M}^{(N-1)}(\beta')}{\mathcal{M}^{(N)}(\beta)}$. This last bound will be more involved and needs a different approach from the previous estimates. It follows the same idea as in \cite{CK} and \cite{Ma} but we will add it here for the sake of completion.
We start by regularizing the potential $(x,y)\mapsto \ln|xy^{-1}|$ by defining the function 
$$V_{\varepsilon}(x,y)=\frac{1}{|B_{\varepsilon}(0)|^{2}}\int_{B_{\varepsilon}(x)}\int_{B_{\varepsilon}(y)}\ln|ab^{-1}| da db.$$
By the Lebesgue differentiation theorem (which holds in the Heisenberg group $\H$), we have that $V_{\varepsilon}(x,y)\to \ln|xy^{-1}|$, for almost every $x,y \in \H$. Next, we define the quantity $\mathcal{M}^{(N)}_{\varepsilon}(\beta)$, by substituting $\ln$ by $V_{\varepsilon}$ in the definition of $\mathcal{M}^{(N)}(\beta)$. We consider the Hilbert space $\mathcal{H}_{\varepsilon}$ obtained by the completion of the set of $C^{\infty}_{0}(\H)$ functions with mean zero, under the dot product $\langle \cdot,\cdot \rangle_{\varepsilon}$ defined by
$$\langle f,g \rangle_{\varepsilon}=-\frac{\beta}{N-1}\int_{\H}\int_{\H}f(x)V_{\varepsilon}(x,y)f(y)dxdy.$$
We also consider the measures $\delta_{y}^{\sharp}\in \mathcal{H}_{\varepsilon}$ defined by
$$\delta_{y}^{\sharp}=\delta_{y}-\chi_{B_{r_{0}}},$$
where $r_{0}$ is picked so that $|B_{r_{0}}|=1$. We introduce the function $W_{\varepsilon}$ and the measure $\tilde{\tau}$ defined by
$$W_{\varepsilon}(x)=\int_{B_{r_{0}}}V_{\varepsilon}(x,y)dy-\frac{1}{2}\int_{B_{r_{0}}}V_{\varepsilon}(x,y)dy,$$
and
$$\tau=e^{\beta W_{\varepsilon}}\tilde{\tau}.$$
With these notations, an easy computation shows that
\begin{align}
\mathcal{M}^{(N)}_{\varepsilon}(\beta)&=\int_{\H^{N}}\exp(-\frac{\beta}{N-1}\sum_{1\leq i<j\leq N}V_{\varepsilon}(x_{i},x_{j}))\prod_{\ell=1}^{N}e^{\beta W_{\varepsilon}(x_{\ell})}\tilde{\tau}(dx_{\ell})\notag\\
&=e^{-\frac{N\beta}{2(N-1)}V_{\varepsilon}(0,0)}\int_{\H^{N}}\exp(\frac{1}{2}\langle \delta^{\sharp}_{(N)},\delta^{\sharp}_{(N)}\rangle_{\varepsilon})\prod_{\ell=1}^{N} \tilde{\tau}(dx_{\ell})\notag
\end{align}
where $\delta_{(N)}^{\sharp}=\sum_{i=1}^{N}\delta_{x_{i}}^{\sharp}$ and where we used the translation invariance of the measure in the Heisenberg group to write $V_{\varepsilon}(x_{i},x_{i})=V_{\varepsilon}(0,0)$. Now using Minlo's theorem for Gaussian functional integration (see \cite{GJ}), we have the existence of a Gaussian average $Ave(\cdot)$ on the space of linear forms $\varphi$, on $\mathcal{H}_{\varepsilon}$, with $Ave(\varphi(\delta_{x}^{\sharp})=0$ and
$$Ave(\varphi(\delta_{x}^{\sharp})\varphi(\delta_{x}^{\sharp}))=\frac{\beta}{N-1}V_{\varepsilon}(x,y).$$
Therefore,
$$Ave(\exp(\varphi(\delta_{(N)}^{\sharp})))=\exp(\frac{1}{2}\langle \delta_{(N)}^{\sharp},\delta_{(N)}^{\sharp}\rangle_{\varepsilon}).$$
Hence,
$$\mathcal{M}_{\varepsilon}^{(N)}(\beta)=e^{-\frac{N\beta}{2(N-1)}V_{\varepsilon}(0,0)}Ave\Big(\big(\int_{\H}\exp(\varphi(\delta_{x}^{\sharp}))\tilde{\tau}(dx)\big)^{N}\Big).$$
Using Jensen's inequality, we have that
$$\mathcal{M}_{\varepsilon}^{(N)}(\beta)\geq \Big(\mathcal{M}_{\varepsilon}^{(N-1)}(\beta')\Big)^{\frac{N}{N-1}}.$$
Thus, after letting $\varepsilon \to 0$, one has
$$\frac{\mathcal{M}^{(N)}(\beta)}{\mathcal{M}^{(N-1)}(\beta')}\geq \Big(\mathcal{M}^{(N-1)}(\beta')\Big)^{\frac{1}{N-1}}.$$
But recall that $\liminf_{N\to \infty}\frac{1}{N}\mathcal{F}^{(N)}(\beta)\geq A_{\beta}$, therefore
$$\liminf_{N\to \infty}\frac{\mathcal{M}^{(N)}(\beta)}{\mathcal{M}^{(N-1)}(\beta')}\geq \mathcal{M}^{(1)} e^{-A_{\beta}},$$
which finishes the proof.
\hfill$\Box$
\begin{proposition}[Uniform Boundedness]
Given $n\geq 1$ and $\beta \in (0,8)$, there exists $N(n,\beta)\in \N$ and a constant $C(n,\beta)$ such that, for $N\geq N(n,\beta)$,
$$\frac{d\mu^{(N)}_{n}}{d\tau^{\otimes n}}\leq C(n,\beta)|R^{(n)}|^{-\frac{\beta}{N-1}}.$$
\end{proposition}
{\it Proof:}
First, we write
$$\frac{d\mu^{(N)}_{n}}{d\tau^{\otimes n}}=K(x_{1},\cdots,x_{n})\frac{|R^{(n)}|^{-\frac{\beta}{N-1}}}{\mathcal{M}^{(N)}(\beta)},$$
where 
$$K(x_{1},\cdots,x_{n})=\int_{\H^{(N-n)}}\prod_{1\leq i\leq n<j\leq N}|x_{i}x_{j}^{-1}|^{-\frac{\beta}{N-1}}\prod_{n\leq k<\ell \leq N}|x_{i}x_{j}^{-1}|^{-\frac{\beta}{N-1}}\tau(dx_{j}).$$
Using H\"{o}lder's inequality, there exists $N(n,\beta)$ such that for $N>N(n,\beta)$ we have
\begin{align}
K(x_{1},\cdots, x_{n})\leq &\Big (\int_{\H^{N-n}}\prod_{1\leq i\leq n<j\leq N}|x_{i}x_{j}^{-1}|^{-\frac{\beta}{2n}}\tau(dx_{j})\Big)^{-\frac{2n}{N-1}}\times \notag\\
&\Big(\int_{\H^{N-n}}\prod_{n\leq i<j \leq N}|x_{i}x_{j}^{-1}|^{-\frac{\beta}{N-1-2n}}\tau(dx_{j})\Big)^{1-\frac{2n}{N-1}}.\notag
\end{align}
For the first term of the right hand side, we have
\begin{align}
\int_{\H^{N-n}}\prod_{1\leq i\leq n<j\leq N}|x_{i}x_{j}^{-1}|^{-\frac{\beta}{2n}}\tau(dx_{j})&=\Big(\int_{\H}\prod_{i=1}^{n}|x_{i}x^{-1}|^{-\frac{\beta}{2n}}\tau(dx)\Big)^{N-n}\notag\\
&\leq \Big(\frac{1}{n}\int_{\H}\sum_{i=1}^{n}|x_{i}x^{-1}|^{-\frac{\beta}{2}}\tau(dx)\Big)^{N-n}\notag\\
&\leq \Big(\sup_{y\in \H}\int_{\H}|yx^{-1}|^{-\frac{\beta}{2}}\tau(dx)\Big)^{N-n}\notag
\end{align}
Hence, the first term is uniformly bounded. For the second term, we first consider 
$$\mathcal{A}_{N}=\Big(\int_{\H^{N-n}}\prod_{n\leq i<j \leq N}|x_{i}x_{j}^{-1}|^{-\frac{\beta}{N-1-2n}}\tau(dx_{j})\Big)^{-\frac{2n}{N-1}}=\Big(\mathcal{M}^{(N-n)}(k(N)\beta)\Big)^{-\frac{2n}{N-1}},$$
where $k(N)=\frac{N-n-1}{N-2n-1}.$ Then clearly $$\limsup_{N\to \infty}\mathcal{A}_{N} \leq \Big(\frac{e^{-\mathcal{A_{\alpha}}}}{\mathcal{M}^{(1)}}\Big).$$
Therefore, in order to finish the proof, one needs to bound $\frac{\mathcal{M}^{(N-n)}(k(N)\beta)}{\mathcal{M}^{(N)}(\beta)}$. Indeed, using Jensen's inequality 
\begin{align}
\frac{\mathcal{M}^{(N-n)}(k(N)\beta)}{\mathcal{M}^{(N)}(\beta)}\leq&\frac{1}{\Big(\mathcal{M}^{(1)}\Big)^{n}}\exp\Big(\frac{n(n-1)}{2(N-1)}\beta \hat{\mu}^{(1)\otimes 2}(\ln|xy^{-1}|)\Big)\times\notag\\
& \exp\Big(n(1-\frac{n}{N-1})\beta \hat{\mu}^{(1)}\otimes \hat{\mu}_{1}^{(N-n),k}(\ln|xy^{-1}|)\Big)\times\notag\\
&\exp\Big(-n(\frac{N-n-1}{N-1})k(N)\beta \hat{\mu}_{2}^{(N-n),k}(\ln|xy^{-1}|)\Big),\notag
\end{align}
where $\mu^{(N-n,k)}$ is defined the same way as $\mn$ with $\beta$ switched with $K(N)\beta$. By Lemma \ref{lem1}, The first exponential term is then bounded uniformly with respect to $N$ and since $k(N)\to 1$ as $N\to \infty$, using the upper bound in  Lemma \ref{lem1} and the upper bound in Lemma \ref{lem2}, we get the uniform boundedness of the the desired quantities.
\hfill$\Box$

The last ingredient for the weak-compactness of the sequence $(\mu_{n}^{(N)})_{n\leq N}$ is its tightness, since we are working in a non-compact domain. So we show the following
\begin{lemma}
The sequence $(\mu_{n}^{(N)})_{n\leq N}$ is tight.
\end{lemma}
{\it Proof:}
Using the symmetry of the measure $\mu^{(N)}_{n}$, it is enough to show tightness for the case $n=1$.  Namely, we need to show that given $\varepsilon>0$, there exists $R(\varepsilon)$ such that
$$\mu_{1}^{(N)}(B_{R(\varepsilon)})\geq 1-\varepsilon.$$
We consider then the map $h:\H\to \R$ defined by
$$h(y)=\int_{\H}\ln|yx^{-1}|\mu^{(1)}(dx)+C,$$
where $C$ is a constant chosen so that $h$ is positive. It is possible to choose such a constant since by construction of $\mu^{(1)}$, $h$ is continuous and $\lim_{y\to \infty} h(y) =+\infty$, uniformly in $y$. Therefore, from Lemma \ref{lem2}, given $\varepsilon>0$, there exists $R(\varepsilon)>0$, such that
$$\hat{\mu}_{1}^{(N)}(h(x))\frac{1}{\varepsilon}\leq \frac{C(\beta)}{\varepsilon}\leq \inf_{x\not \in B_{R(\varepsilon)}}h(x).$$ 
Thus,
\begin{align}
\hat{\mu}^{(N)}_{1}(h(x))&\geq \hat{\mu}^{(N)}_{1}(h(x)\chi_{\H\setminus B_{R(\varepsilon)}})\notag\\
&\geq \frac{1}{\varepsilon} \hat{\mu}_{1}^{(N)}(h(x)) \hat{\mu}^{(N)}_{1}(\chi_{\H\setminus B_{R(\varepsilon)}})\notag \\
&\geq \frac{1}{\varepsilon} \hat{\mu}_{1}^{(N)}(h(x)) (1-\mu^{(N)}_{1}(B_{R(\varepsilon)})).\notag
\end{align}
The result then follows after dividing by $\frac{1}{\varepsilon} \hat{\mu}_{1}^{(N)}(h(x))$.
\hfill$\Box$

Now given the weak compactness, the rest of the procedure of Section 3 can be carried out to prove the following
\begin{theorem}
Given a function $Q$ satisfying $(a)$ and $(b)$. Then, for any $\beta \in (0,8)$, there exists $\rho_{\beta}\in L^{p}(\H)$ for all $p\geq 1$, such that
$$\rho_{\beta}(x)=\frac{Q(x)e^{K(x)-\beta \int_{\H}\ln|xy^{-1}|\rho_{\beta}(y)dy}}{\int_{\H}Q(x)e^{K(x)-\beta \int_{\H}\ln|xy^{-1}|\rho_{\beta}(y)dy}dx}.$$
\end{theorem}
Theorem \ref{thmh} and Corollary \ref{cor} are a direct corollary of the previous theorem.

\end{document}